\documentclass[11pt,a4paper]{article}
\usepackage{amsmath,amssymb,amsthm}
\usepackage[margin=2.5cm]{geometry}
\usepackage{hyperref}

\newcommand{\osp}{\mathfrak{osp}}
\newcommand{\Fc}{\mathcal{F}}
\newcommand{\Dc}{\mathcal{D}}
\newcommand{\Sc}{\mathcal{S}}
\newcommand{\Kc}{\mathcal{K}}
\newcommand{\Ec}{\mathcal{E}}
\newcommand{\Ehat}[1]{\widehat{\Ec}_{#1}}
\newcommand{\Hd}{\mathrm{H}}
\newcommand{\Zc}{\mathrm{Z}}
\newcommand{\Bc}{\mathrm{B}}
\newcommand{\R}{\mathbb{R}}
\newcommand{\N}{\mathbb{N}}
\newcommand{\Zz}{\mathbb{Z}}
\newcommand{\Cc}{\mathbb{C}}
\newcommand{\Q}{\mathbb{Q}}

\theoremstyle{plain}
\newtheorem{theorem}{Theorem}[section]
\newtheorem{proposition}[theorem]{Proposition}
\newtheorem{lemma}[theorem]{Lemma}
\newtheorem{corollary}[theorem]{Corollary}
\theoremstyle{definition}
\newtheorem{definition}[theorem]{Definition}
\newtheorem{remark}[theorem]{Remark}

\title{Deformations of the $\osp(n|2)$-action on the superspace\\ of symbols of differential operators on $\R^{1|n}$}
\author{I. Basdouri\thanks{Universit\'e de Gafsa, Facult\'e des Sciences, D\'epartement de Math\'ematiques. E-mail: basdourimed@yahoo.fr} \and
M. Ben Ammar\thanks{Universit\'e de Sfax, Facult\'e des Sciences, D\'epartement de Math\'ematiques. E-mail: mabrouk.benammar@gmail.com}}
\date{}

\begin{document}
\maketitle

\begin{abstract}
We study formal deformations of the natural $\osp(n|2)$-action, $n\geq 3$, on the superspace
$\Sc^n_d=\bigoplus_{k\geq 0}\Fc^n_{d-\frac{k}{2}}$ of symbols of linear differential operators on
weighted densities over $\R^{1|n}$. Starting from the first cohomology space computed in
\cite{10}, we compute the cup-product
$\Hd^1\vee \Hd^1\to \Hd^2$ which carries the quadratic obstructions. The answer is governed by the
$\osp(n|2)$-invariant operators $A_k=\eta_1\cdots\eta_n\partial_x^{k-1}$: the two cocycles $h_k$ and
$\widetilde{h}_k$ spanning the off-diagonal part of $\Hd^1$ are exactly the two derivatives of the
coboundary of $A_k$ with respect to the two weights. Consequently, all the products of two
off-diagonal classes and all the products of two diagonal classes vanish, and the whole obstruction is
carried, for each $k$, by a single non-trivial 2-cocycle $\Omega_k$. If $2d\notin\N$ the space
$\Hd^1\vee\Hd^1$ is identically zero, so every infinitesimal deformation is integrable. If $2d=m\in\N$
we obtain exactly $m$ quadratic integrability conditions,
$\tau_{2-n-k}(t_k-\widetilde{t}_k)+\tau_k\widetilde{t}_k=0$, $1\leq k\leq m$, and we prove that they are
also sufficient: no condition of order $\geq 3$ occurs and the versal deformation is of degree one in
the parameters. In particular every integrable formal deformation is equivalent to its infinitesimal
part.


\end{abstract}

\section{Introduction}

Let $\Fc_\lambda$ be the space of weighted densities of weight $\lambda$ on the line and
$\Dc_{\lambda,\mu}=\mathrm{Hom}_{\mathrm{diff}}(\Fc_\lambda,\Fc_\mu)$ the corresponding
$\mathrm{Vect}(\R)$-module of linear differential operators. Lecomte \cite{14} computed
$\Hd^1_{\mathrm{diff}}(\mathfrak{sl}(2),\Dc_{\lambda,\mu})$ and
$\Hd^2_{\mathrm{diff}}(\mathfrak{sl}(2),\Dc_{\lambda,\mu})$; these spaces control, respectively, the
infinitesimal deformations of the $\mathfrak{sl}(2)$-module of symbols
$\Sc_{\mu-\lambda}=\bigoplus_{k\geq0}\Fc_{\mu-\lambda-k}$ and the obstructions to integrating them, in
the sense of Nijenhuis and Richardson \cite{15}. Multi-parameter and versal deformations of these
$\mathrm{Vect}(\R)$- and $\mathfrak{sl}(2)$-modules, and of the modules of differential forms, were
obtained in \cite{1,2}.

The super analogue of this picture replaces $\R$ by the superspace $\R^{1|n}$ with its standard contact
structure, $\mathrm{Vect}(\R)$ by the Lie superalgebra $\Kc(n)$ of contact vector fields, and
$\mathfrak{sl}(2)$ by the orthosymplectic Lie superalgebra $\osp(n|2)\subset\Kc(n)$; the corresponding
symbol calculus and conformally equivariant quantization were developed in \cite{13,11}. The corresponding
first cohomology space $\Hd^1_{\mathrm{diff}}(\osp(n|2),\Dc^n_{\lambda,\mu})$ was computed by
I. Basdouri and M. Ben Ammar \cite{4} for $n=1$, by N. Ben Fraj and M. Boujelben \cite{9} for $n=2$ (see also \cite{8} for the whole contact
superalgebra $\Kc(2)$), and by
N. Ben Fraj, A. Jabeur and I. Safi for all $n\geq3$ \cite{10}, together with explicit spanning cocycles. The
present paper takes that computation as its input and carries out the next step: the deformation theory
itself. On the deformation side, the low-rank cases have already received separate treatment:
I. Basdouri and M. Ben Ammar \cite{5} studied deformations of the $\mathfrak{sl}(2)$- and
$\osp(1|2)$-modules of symbols, and M. Ben Ammar and W. Mtaouaa \cite{6} studied deformations of
$\osp(2|2)$-modules of weighted densities on $\R^{1|2}$; the present work extends this line of
investigation to arbitrary $n\geq3$.

More precisely, we prove the following. Write $d=\mu-\lambda$ and let
$\Sc^n_d=\bigoplus_{k\geq0}\Fc^n_{d-\frac{k}{2}}$ be the space of symbols, on which $\osp(n|2)$ acts
diagonally by the Lie derivative $L$. The space of infinitesimal deformations of $L$ is
$\Hd^1(\osp(n|2),\Dc^n_d)$, $\Dc^n_d=\mathrm{Hom}_{\mathrm{diff}}(\Sc^n_d,\Sc^n_d)$, and it is spanned
by the ``diagonal'' cocycles $f_\nu(X_F)=F'$, one for each weight $\nu$ occurring in $\Sc^n_d$,
together with, when $2d=m\in\N$, the $2m$ ``off-diagonal'' cocycles $h_k,\widetilde{h}_k$
($1\leq k\leq m$) of \cite{10}, which map the summand of weight
$\lambda_k=-\frac{k+n-2}{2}$ into the summand of weight $\mu_k=\frac{k}{2}$.

\begin{itemize}
\item[(a)] The whole cup-product is computed by the following mechanism (Section 5). At the resonant
pair $(\lambda_k,\mu_k)$ the operator $A_k=\eta_1\cdots\eta_n\partial_x^{k-1}$ is
$\osp(n|2)$-invariant, and if one deforms the two weights into $(\lambda_k+\sigma,\mu_k+\tau)$ then the
coboundary of $A_k$ becomes $\sigma a^k_1+\tau a^k_2$ with
$$a^k_2=h_k,\qquad a^k_1=-(h_k+\widetilde{h}_k).$$
Thus the two spanning cocycles of \cite{10} are nothing but the two weight-derivatives of an invariant
operator. Since $\delta^2=0$ identically in $(\sigma,\tau)$, this forces
$f_{\lambda_k}\vee a^k_1=0$, $f_{\mu_k}\vee a^k_2=0$ and
$f_{\lambda_k}\vee a^k_2=-f_{\mu_k}\vee a^k_1=:\Omega_k$.
\item[(b)] Every product of two off-diagonal classes vanishes, and so does every product of two
diagonal classes; the products of a diagonal by an off-diagonal class vanish unless the weight of the
diagonal factor is $\lambda_k$ or $\mu_k$. Hence $\Hd^1\vee\Hd^1$ is spanned by the $m$ classes
$[\Omega_k]$, which are non-zero (Proposition~\ref{res:6.2}; the verification of non-triviality is
partly computer-assisted, see Remark~\ref{res:6.3}).
\item[(c)] If $2d\notin\N$, then $\Hd^1\vee\Hd^1=0$: there is no obstruction at all, and the versal
deformation is the infinitesimal one.
\item[(d)] If $2d=m\in\N$, the quadratic integrability conditions are exactly the $m$ relations
$\tau_{2-n-k}(t_k-\widetilde{t}_k)+\tau_k\widetilde{t}_k=0$, $1\leq k\leq m$, and these conditions are
sufficient: on the variety they define, the deformation $L+L_1$, with no term of order $\geq 2$ at all,
is already a homomorphism of Lie superalgebras. Consequently no integrability condition of order
$\geq3$ appears and every formal deformation is equivalent to its infinitesimal part.
\end{itemize}

We emphasise that (d) corrects the naive expectation that the quadratic conditions vanish
identically for all $d$: they do vanish for $2d\notin\N$, but for $2d=m\in\N$ the classes
$[\Omega_k]$ are genuine obstructions and cut out a quadric in the parameter space (see
Remark~\ref{res:7.3}).

\section{Notation and the modules involved}

\subsection{The superspace $\R^{1|n}$ and contact vector fields}

Let $\R^{1|n}$ be the superspace with coordinates $(x,\theta)$,
$\theta=(\theta_1,\dots,\theta_n)$, where the $\theta_i$ are odd:
$\theta_i\theta_j=-\theta_j\theta_i$. The standard contact structure is given by
the $1$-form
\begin{equation}\label{eq:alpha}
\omega_n=dx+\sum_{i=1}^n\theta_i\,d\theta_i .
\end{equation}
On $C^\infty(\R^{1|n})$ one has the contact bracket
\begin{equation}\label{eq:2}
\{F,G\}=FG'-F'G-\tfrac12(-1)^{F}\sum_{i=1}^n\overline{\eta}_i(F)\cdot\overline{\eta}_i(G)=FG'-F'G+\tfrac12\sum_{i=1}^n\eta_i(F)\cdot\overline{\eta}_i(G),
\end{equation}
where
$$
F'=\partial_xF,\quad \eta_i=\frac{\partial}{\partial\theta_i}+\theta_i\frac{\partial}{\partial x}\quad\text{and}\quad \overline{\eta}_i=\frac{\partial}{\partial\theta_i}-\theta_i\frac{\partial}{\partial x}.
$$
 The odd derivations
$\eta_i$ generate the kernel of \eqref{eq:alpha} as a module over
$C^\infty(\R^{1|n})$ and satisfy
\begin{equation}\label{eq:3}
\eta_i\eta_j+\eta_j\eta_i=2\delta_{ij}\partial_x,\quad
\text{in particular}\quad \eta_i^2=\partial_x .
\end{equation}
The Lie superalgebra $K(n)$ of contact vector fields on $\R^{1|n}$ is spanned by
\[
X_F=F\partial_x-\frac12(-1)^F\sum\eta_i(F)\eta_i=F\partial_x+\frac12\sum_{i=1}^n\eta_i(F)\,\overline{\eta}_i,
\quad F\in C^\infty(\R^{1|n}),
\]
where the braket is given by: $$[X_F,X_G]=X_{\{F,G\}}.$$ The orthosymplectic Lie superalgebra is realised as the
finite-dimensional subalgebra
\begin{equation}
\label{eq:4}
\osp(n|2)=\mathrm{Span}\bigl(X_1,\;X_x,\;X_{x^2},\;X_{\theta_i},\;X_{x\theta_i},\;X_{\theta_i\theta_j}\bigr),
\qquad 1\leq i<j\leq n,
\end{equation}
with even part containing the copy of $\mathfrak{sl}(2)$ spanned by $X_1,X_x,X_{x^2}$ and the copy of
$\mathfrak{so}(n)$ spanned by the $X_{\theta_i\theta_j}$; thus
$\dim\osp(n|2)=\bigl(3+\tfrac{n(n-1)}{2}\bigr)+2n$.

We shall use repeatedly the following elementary consequence of \eqref{eq:2}.

\begin{lemma}
\label{res:2.1}
The Lie superalgebra $\osp(n|2)$ is generated by its odd part, and even by the $2n$ elements
$X_{\theta_i},X_{x\theta_i}$, $1\leq i\leq n$. Indeed
$$\{\theta_i,\theta_i\}=\tfrac12,\qquad \{\theta_i,x\theta_i\}=\tfrac{x}{2},\qquad
\{x\theta_i,x\theta_i\}=\tfrac{x^2}{2},\qquad \{x\theta_i,\theta_j\}=-\tfrac12\theta_i\theta_j\ (i\neq j).$$
\end{lemma}

\begin{proof}
Immediate from \eqref{eq:2}, using $\overline{\eta}_k(\theta_i)=\delta_{ki}$ and
$\overline{\eta}_k(x\theta_i)=x\delta_{ki}-\theta_k\theta_i$.
\end{proof}

\subsection{Densities, operators and symbols}

The space of $\lambda$-densities is
\begin{equation}
\label{eq:5}
\Fc^n_\lambda=\bigl\{F\omega_n^\lambda\ |\ F\in C^\infty(\R^{1|n})\bigr\},
\end{equation}
a $\Kc(n)$-module (hence an $\osp(n|2)$-module) for the Lie derivative
\begin{equation}
\label{eq:6}
L_{X_F}\bigl(G\omega_n^\lambda\bigr)=L^\lambda_{X_F}(G)\,\omega^\lambda_n,\qquad
L^\lambda_{X_F}=X_F+\lambda F'.
\end{equation}
The superspace of linear differential operators
$\Dc^n_{\lambda,\mu}=\mathrm{Hom}_{\mathrm{diff}}(\Fc^n_\lambda,\Fc^n_\mu)$ is a $\Kc(n)$-module for
\begin{equation}
\label{eq:7}
X_F\cdot A=L^\mu_{X_F}\circ A-(-1)^{AF}A\circ L^\lambda_{X_F}.
\end{equation}
By \eqref{eq:3}, and agreeing that $\partial_x^{\frac12}=\eta_n$, the superspace $\Dc^n_{\lambda,\mu}$ is
spanned, as a $C^\infty(\R^{1|n})$-module, by the operators $D^\alpha$,
$\alpha=(\alpha_1,\dots,\alpha_n)\in\{0,1\}^{n-1}\times\frac12\N$, defined by
$$D^\alpha(F\omega^\lambda_n)=\eta_1^{\alpha_1}\cdots\eta_{n-1}^{\alpha_{n-1}}\partial_x^{\alpha_n}(F)\,\omega^\mu_n .$$
The module $\Dc^n_{\lambda,\mu}$ is filtered,
$$\Dc^{n,0}_{\lambda,\mu}\subset\Dc^{n,\frac12}_{\lambda,\mu}\subset\Dc^{n,1}_{\lambda,\mu}\subset
\Dc^{n,\frac32}_{\lambda,\mu}\subset\cdots\subset\Dc^{n,\ell}_{\lambda,\mu}\subset\cdots,$$
where $\Dc^{n,k}_{\lambda,\mu}$ is the submodule of  $\Dc^n_{\lambda,\mu}$ spanned, as a $C^\infty(\R^{1|n})$-module, by
$\{D^\alpha,\ |\alpha|\leq k\}$, where $|\alpha|=\frac12(\alpha_1+\cdots+\alpha_{n-1})+\alpha_n.$ is the
order of $D^\alpha$. The graded module $\Sc_{\lambda,\mu}:=\mathrm{gr}\,\Dc^n_{\lambda,\mu}$ is called
the space of symbols.

\begin{proposition}
\label{res:2.2} For any fixed $\alpha=(\alpha_1,\dots,\alpha_n)\in\{0,1\}^{n-1}\times\frac 12\mathbb{N}$, with $|\alpha|=k$, the submodule of  $\Dc^{n,k}_{\lambda,\mu}/\Dc^{n,k-\frac{1}{2}}_{\lambda,\mu}$ spanned, as a $C^\infty(\mathbb{R}^{1|n})$-module, by the class of $D^\alpha$ is isomorphic to $\Fc_{\mu-\lambda-k}^n$. The isomorphism is induced from the following map:
$$
a(x,\theta)D^\alpha\mapsto a(x,\theta)\omega_n^{\mu-\lambda-k}.
$$
\end{proposition}

\begin{proof}
Let $\mathcal{M}^\alpha=\{a\,D^\alpha:a\in C^\infty(\R^{1|n})\}$ be the corresponding subspace of the
graded module, i.e.\ we work modulo operators of order $<k$. By \eqref{eq:7},
$$X_F\cdot(aD^\alpha)=L^\mu_{X_F}\circ(aD^\alpha)-(-1)^{(a+|\alpha|)F}(aD^\alpha)\circ L^\lambda_{X_F},$$
and each of the two terms is the sum of $aD^\alpha\cdot(\text{terms of }L)$ plus operators obtained by
letting the derivatives of $L$ fall on $a$ or on the coefficients of $D^\alpha$, which lowers the order.
Modulo operators of order $<k$ only the ``leading'' contributions survive; a direct computation using
$X_F=F\partial_x-\frac12(-1)^F\sum\eta_i(F)\eta_i$, $\partial_x\circ D^\alpha=D^\alpha\circ\partial_x$
and $\eta_i\circ D^\alpha=(-1)^{|\alpha|}D^\alpha\circ \eta_i$ modulo lower order, gives
$$X_F\cdot(aD^\alpha)\equiv\Bigl(X_F(a)+(\mu-\lambda-k)F'a\Bigr)D^\alpha
=\Bigl(L^{\mu-\lambda-k}_{X_F}(a)\Bigr)D^\alpha \pmod{\text{order}<k},$$
because the total number of derivatives $\partial_x$ produced when $L^\lambda_{X_F}$ crosses $D^\alpha$
is $-k$ times $F'$, and the two weights contribute $\mu$ and $-\lambda$. Comparing with \eqref{eq:6} this is
precisely the action on $\Fc^n_{\mu-\lambda-k}$.
\end{proof}

\begin{corollary}
\label{res:2.3}
The quotient-module $\Dc^{n,k}_{\lambda,\mu}/\Dc^{n,k-\frac12}_{\lambda,\mu}$, $k\in\frac12\N$, is
isomorphic to a direct sum of $N$ copies of $\Fc^n_{\mu-\lambda-k}$, where $N=2^{n-1}$ if $2k\geq n-1$ or
$N=\sum_{i=0}^{2k}\binom{n-1}{i}$ if $2k<n-1$. In particular, we have
$\Dc^{1,k}_{\lambda,\mu}/\Dc^{1,k-\frac12}_{\lambda,\mu}\simeq\Fc^1_{\mu-\lambda-k}$.
\end{corollary}

\begin{proof} The module $\Dc^{n,k}_{\lambda,\mu}/\Dc^{n,k-\frac{1}{2}}_{\lambda,\mu}$ is spanned by the classes of the $D^\alpha$ with $|\alpha|=k$ and by Proposition~\ref{res:2.2} each of these classes contributes one copy of
$\Fc^n_{\mu-\lambda-k}$, and these copies are independent in the quotient; it remains to count the
$\alpha=(\alpha_1,\dots,\alpha_n)\in\{0,1\}^{n-1}\times\frac12\N$ with
$|\alpha|=\frac12(\alpha_1+\cdots+\alpha_{n-1})+\alpha_n=k$. Choosing the odd part
$(\alpha_1,\dots,\alpha_{n-1})$ freely determines $\alpha_n=k-\frac12(\alpha_1+\cdots+\alpha_{n-1})$,
which is admissible precisely when $\alpha_1+\cdots+\alpha_{n-1}\leq 2k$; this gives $2^{n-1}$ choices as
soon as $2k\geq n-1$, and $\sum_{i=0}^{2k}\binom{n-1}{i}$ choices otherwise.
\end{proof}

\begin{corollary}
\label{res:2.4}
As a $\Kc(n)$-module, the module of symbols of $\Dc^n_{\lambda,\mu}$ depends only on the difference
$d=\mu-\lambda$, and it is isomorphic to a direct sum of copies of
\begin{equation}
\label{eq:8}
\Sc^n_d=\bigoplus_{k\geq0}\Fc^n_{d-\frac{k}{2}},\qquad d=\mu-\lambda .
\end{equation}
\end{corollary}

The summand $\Fc^n_{d-\frac k2}$ is the space of symbols of order $\frac k2$, the half-integer step
reflecting the fact that the odd derivations $\eta_i$ have order $\frac12$ with respect to
$\partial_x=\eta_i^2$. One has
\begin{equation}
\label{eq:9}
\mathrm{End}(\Sc^n_d)\supset\bigoplus_{i,j\geq0}\mathrm{Hom}_{\mathrm{diff}}\bigl(\Fc^n_{d-\frac j2},\Fc^n_{d-\frac i2}\bigr)
=\bigoplus_{i,j\geq0}\Dc^n_{d-\frac j2,\,d-\frac i2},
\end{equation}
and in this paper we study the deformations of the action of $\osp(n|2)$ on
$\Dc^n_d:=\bigoplus_{i,j\geq0}\Dc^n_{d-\frac j2,\,d-\frac i2}$.

\section{The cohomological input}

Let $\mathfrak{g}$ be a Lie superalgebra acting on a superspace $V$. The space of $m$-cochains is
$C^m(\mathfrak{g},V)=\mathrm{Hom}(\Lambda^m\mathfrak{g},V)$, with coboundary
$\delta:C^m(\mathfrak{g},V)\to C^{m+1}(\mathfrak{g},V)$, $\delta^2=0$. We need $\delta$ in degrees $0$
and $1$:
\begin{equation}
\label{eq:10}
\delta u(x)=(-1)^{ux}u\cdot x, \quad x\in V,
\end{equation}
\begin{equation}
\label{eq:11}
\delta f(u,v)=(-1)^{uf}u\cdot f(v)
-(-1)^{v(u+f)}v\cdot f(v)-f([u,v]).
\end{equation}
We write $\Zc^m,\Bc^m,\Hd^m$ for cocycles, coboundaries and cohomology, adding the subscript
$\mathrm{diff}$ when only cochains given by differential operators are considered.

The following is the main result of \cite{10} (for $n=1,2$ see \cite{4,9}).

\begin{theorem}[\cite{10}]
\label{res:3.1}
For $n\geq3$,
\begin{equation}
\label{eq:12}
\Hd^1_{\mathrm{diff}}\bigl(\osp(n|2),\Dc^n_{\lambda,\mu}\bigr)\simeq
\begin{cases}
\R & \text{if }\mu-\lambda=0,\\[2pt]
\R^2 & \text{if }(\lambda,\mu)=\bigl(-\frac{k+n-2}{2},\frac{k}{2}\bigr),\ k\in\N\setminus\{0\},\\[2pt]
0 & \text{otherwise},
\end{cases}
\end{equation}
the corresponding cohomology classes being spanned by the $1$-cocycles
\begin{align}
f_\lambda(X_F)&=F',\label{eq:13}\\
h_k(X_F)&=(-1)^{nF}F'\,\eta_1\cdots\eta_n\,\partial_x^{k-1},\label{eq:14}\\
\widetilde{h}_k(X_F)&=(k-1)F''\eta_1\cdots\eta_n\partial_x^{k-2}
+(-1)^{(n-1)F}\sum_{\sigma\in S_n}\frac{\mathrm{sgn}(\sigma)}{(n-1)!}\eta_{\sigma(1)}(F')\,
\eta_{\sigma(2)}\cdots\eta_{\sigma(n)}\partial_x^{k-1}.\label{eq:15}
\end{align}
The cocycles $f_\lambda$ are even, while $h_k$ and $\widetilde{h}_k$ have the same parity as $n$.
\end{theorem}

Since $\mathrm{sgn}(\sigma)\eta_{\sigma(1)}\cdots\eta_{\sigma(n)}=\eta_1\cdots\eta_n$, then the cocycle $\widetilde{h}_k$ may be rewritten as
\begin{equation}
\label{eq:16}
\widetilde{h}_k(X_F)=(k-1)F''\,\eta_1\cdots\eta_n\,\partial_x^{k-2}
+(-1)^{(n-1)F}\sum_{i=1}^n(-1)^{i-1}\eta_i(F')\,\eta_1\cdots\check{\eta_i}\cdots\eta_n\,\partial_x^{k-1},
\end{equation}
which is the form we shall use.

In \cite{10} only differential cocycles were considered; the following shows that nothing is lost. The
argument is the standard reduction of cochains for a Lie superalgebra containing an element acting
surjectively, cf.\ \cite{12}. Let
$M$ be the submodule of $\Dc^n_{\lambda,\mu}$ consisting of the
$\sum_\alpha a_\alpha(x,\theta)D^\alpha$ with polynomial $a_\alpha$.

\begin{theorem}
\label{res:3.2}
For any $n\geq1$,
$\Hd^1_{\mathrm{diff}}(\osp(n|2),\Dc^n_{\lambda,\mu})=\Hd^1(\osp(n|2),\Dc^n_{\lambda,\mu})
=\Hd^1(\osp(n|2),M)$. Moreover, up to a coboundary, any $1$-cocycle
$c\in \Zc^1(\osp(n|2),\Dc^n_{\lambda,\mu})$ is reduced, i.e.\ $c(X_1)=0$.
\end{theorem}

\begin{proof}
The case $n=0$ is due to D. Arnal and M. Ben Ammar \cite{3}. Observe that $X_1\cdot$ is the derivation with
respect to $x$: $X_1\cdot\bigl(a_\alpha(x,\theta)D^\alpha\bigr)=\partial_x a_\alpha(x,\theta)D^\alpha$.
Therefore $X_1\cdot$ is surjective and
$M=\{A\in\Dc^n_{\lambda,\mu}:\exists k\in\N,\ (X_1\cdot)^kA=0\}$. If $u\mapsto a(u)$ is a $1$-cochain,
there is $b\in\Dc^n_{\lambda,\mu}$ with $X_1\cdot b=a(X_1)$; if $a$ is a cocycle then $c=a-\delta b$ is
reduced. The cocycle relations on the pairs $(X_1,X_x)$, $(X_1,X_{x^2})$, $(X_1,X_{\theta_i})$,
$(X_1,X_{x\theta_i})$, $(X_1,X_{\theta_i\theta_j})$ then give
$$X_1\cdot c(X_x)=0,\quad X_1^2\cdot c(X_{x^2})=0,\quad X_1\cdot c(X_{\theta_i})=0,\quad
X_1^2\cdot c(X_{x\theta_i})=0,\quad X_1\cdot c(X_{\theta_i\theta_j})=0,$$
so that all the values of $c$ lie in $M$. Observe that $b$ is differential, and that if $c$ takes its
values in $M$ we may choose $b\in M$; hence
\begin{align*}
\Zc^1(\osp(n|2),\Dc^n_{\lambda,\mu})&=\Zc^1_{\mathrm{red}}(\osp(n|2),M)+\Bc^1(\osp(n|2),\Dc^n_{\lambda,\mu}),\\
\Zc^1_{\mathrm{diff}}(\osp(n|2),\Dc^n_{\lambda,\mu})&=\Zc^1_{\mathrm{red}}(\osp(n|2),M)+\Bc^1_{\mathrm{diff}}(\osp(n|2),\Dc^n_{\lambda,\mu}),\\
\Zc^1_{\mathrm{diff}}(\osp(n|2),M)&=\Zc^1_{\mathrm{red}}(\osp(n|2),M)+\Bc^1_{\mathrm{diff}}(\osp(n|2),M).
\end{align*}
By definition $\Bc^1(\osp(n|2),M)$ is contained in
$\Zc^1(\osp(n|2),\Dc^n_{\lambda,\mu})\cap\Bc^1(\osp(n|2),\Dc^n_{\lambda,\mu})$ and in
$\Zc^1(\osp(n|2),\Dc^n_{\lambda,\mu})\cap\Bc^1_{\mathrm{diff}}(\osp(n|2),\Dc^n_{\lambda,\mu})$. For the
reverse inclusion, let $c\in\Zc^1_{\mathrm{diff}}(\osp(n|2),M)$ and $b\in\Dc^n_{\lambda,\mu}$ with
$c=\delta b$. Then $X_1\cdot b=c(X_1)\in M$, hence $b\in M$ and
$c\in\Bc^1_{\mathrm{diff}}(\osp(n|2),M)$. Thus
$$\Bc^1(\osp(n|2),M)=\Zc^1(\osp(n|2),M)\cap\Bc^1(\osp(n|2),\Dc^n_{\lambda,\mu})
=\Zc^1(\osp(n|2),M)\cap\Bc^1_{\mathrm{diff}}(\osp(n|2),\Dc^n_{\lambda,\mu}),$$
which proves the three announced equalities.
\end{proof}

\begin{remark}
\label{res:3.3}
The same argument applies in any degree: for any $n\geq1$ and any $k\geq0$,
$\Hd^k_{\mathrm{diff}}(\osp(n|2),\Dc^n_{\lambda,\mu})=\Hd^k(\osp(n|2),\Dc^n_{\lambda,\mu})
=\Hd^k(\osp(n|2),M)$, and, up to a coboundary, any $k$-cocycle is reduced. We use this for $k=2$ in
Section 6.
\end{remark}

\section{Deformations: the formalism}

Throughout this section $\mathfrak{g}$ is a Lie superalgebra and $V$ a $\mathfrak{g}$-module with
action $L:\mathfrak{g}\to\mathrm{End}(V)$.

\begin{definition}
\label{res:4.1}
A \emph{formal deformation} of $L$ is a formal series
\begin{equation}
\label{eq:17}
\widetilde{L}=L+\sum_{i\geq1}t^iL_i,\qquad L_i\in C^1\bigl(\mathfrak{g},\mathrm{End}(V)\bigr)_{\bar0},
\end{equation}
with $t$ a formal even parameter, such that $\widetilde{L}$ is a homomorphism of Lie superalgebras into
$\mathrm{End}(V)[[t]]$, i.e.
\begin{equation}
\label{eq:18}
[\widetilde{L}_X,\widetilde{L}_Y]=\widetilde{L}_{[X,Y]}\quad\text{for all }X,Y\in\mathfrak{g}.
\end{equation}
It is \emph{infinitesimal} if $L_i=0$ for $i\geq2$, understood modulo $t^2$. Two formal deformations
$\widetilde{L},\widetilde{L}'$ are \emph{equivalent} if there is a formal inner automorphism
\begin{equation}
\label{eq:19}
\mathcal{A}=\mathrm{id}+\sum_{i\geq1}t^i\mathcal{A}_i,\qquad \mathcal{A}_i\in\mathrm{End}(V)_{\bar0},
\end{equation}
with $\widetilde{L}'=\mathcal{A}\circ\widetilde{L}\circ\mathcal{A}^{-1}$. The deformation is
\emph{trivial} if it is equivalent to $L$.
\end{definition}

Expanding \eqref{eq:18} in powers of $t$ gives, in degree $1$,
\begin{equation}
\label{eq:20}
\delta L_1=0,
\end{equation}
so $L_1\in \Zc^1(\mathfrak{g},\mathrm{End}(V))$; moreover if $L_1-L_1'\in\Bc^1$ the corresponding
infinitesimal deformations are equivalent, so that $\Hd^1(\mathfrak{g},\mathrm{End}(V))$ classifies the
infinitesimal deformations. In degree $m\geq2$,
\begin{equation}
\label{eq:21}
\delta L_m+\frac12\sum_{i+j=m,\ i,j\geq1}L_i\vee L_j=0,
\end{equation}
where $\vee$ is the cup-product, defined for arbitrary linear maps
$a,b:\mathfrak{g}\to\mathrm{End}(V)$ by
\begin{equation}
\label{eq:22}
(a\vee b)(x,y)=(-1)^{xb}[a(x),b(y)]+(-1)^{a(x+b)}[b(x),a(y)],
\end{equation}
so that $(a\vee b)(x,y)=[a(x),b(y)]+[b(x),a(y)]$ when $a$ and $b$ are even. The first non-trivial
relation
$$\delta L_2+\tfrac12 (L_1\vee L_1)=0$$
is the first obstruction to integrating an infinitesimal deformation: $L_1\vee L_1$ must be a
coboundary.

For any two $1$-cocycles $C_1,C_2\in \Zc^1(\mathfrak{g},\mathrm{End}(V))$ the bilinear map
$C_1\vee C_2$ is a $2$-cocycle, and it is a $2$-coboundary as soon as one of $C_1,C_2$ is a coboundary.
Hence \eqref{eq:22} induces a bilinear map
\begin{equation}
\label{eq:23}
\Hd^1(\mathfrak{g},\mathrm{End}(V))\otimes \Hd^1(\mathfrak{g},\mathrm{End}(V))\longrightarrow
\Hd^2(\mathfrak{g},\mathrm{End}(V)).
\end{equation}
All the obstructions lie in $\Hd^2$ and, at the quadratic level, in the image of the cup-product.

If $\dim \Hd^1(\mathfrak{g},V)=m$, choose $1$-cocycles $\omega_1,\dots,\omega_m$ representing a basis
and consider
\begin{equation}
\label{eq:24}
\rho=\rho_0+\sum_{i=1}^m t_i\omega_i ,
\end{equation}
where $t_1,\dots,t_m$ are independent parameters with $|t_i|=|\omega_i|$. One tries to extend \eqref{eq:24} to a
formal deformation
\begin{equation}
\label{eq:25}
\rho=\rho_0+\sum_{i=1}^m t_i\omega_i+\sum_{i,j}t_it_j\rho^2_{ij}+\cdots,
\qquad [\rho(x),\rho(y)]=\rho([x,y]),
\end{equation}
with $|\rho^2_{ij}|=|t_it_j|$, etc. The obstructions force extra algebraic relations on
$t_1,\dots,t_m$; if $R\subset\Cc[[t_1,\dots,t_m]]$ denotes the ideal they generate, the base of the
versal deformation is the supercommutative associative superalgebra with unity
\begin{equation}
\label{eq:26}
\mathcal{A}=\Cc[[t_1,\dots,t_m]]/R .
\end{equation}

We now specialise to $\mathfrak{g}=\osp(n|2)$, $V=\Sc^n_d$, $n\geq3$. By \eqref{eq:9} the infinitesimal
deformations are described by
$$\Hd^1(\osp(n|2),\Dc^n_d)=\bigoplus_{i,j\geq0}\Hd^1\bigl(\osp(n|2),\Dc^n_{d-\frac j2,\,d-\frac i2}\bigr),$$
and by Theorem~\ref{res:3.1} the summand is non-zero only if $i=j$, or
$\bigl(d-\frac j2,d-\frac i2\bigr)=\bigl(-\frac{k+n-2}{2},\frac k2\bigr)$ for some
$k\in\N\setminus\{0\}$. The latter forces $i=2d-k$ and $j=2d+k+n-2$, so it can occur only if
$2d\in\N$ and $1\leq k\leq 2d$. Hence two cases.

\begin{theorem}
\label{res:4.2}
Let $n\geq3$ and $\Sc^n_d$ as in \eqref{eq:8}.
\begin{itemize}
\item[i)] If $2d\notin\N$, then
$\Hd^1(\osp(n|2),\Dc^n_d)=\bigoplus_{k\geq0}\Hd^1\bigl(\osp(n|2),\Dc^n_{d-\frac k2,\,d-\frac k2}\bigr)$
and the general infinitesimal deformation of $\Sc^n_d$ is
\begin{equation}
\label{eq:27}
\widetilde{L}=L+L_1,\qquad L_1=\sum_{k\geq0}\tau_k\,f_{d-\frac k2},
\end{equation}
the coefficients $\tau_k$ being even independent parameters.
\item[ii)] If $2d=m\in\N$, then
$$\Hd^1(\osp(n|2),\Dc^n_d)=\bigoplus_{k=1}^m \Hd^1\bigl(\osp(n|2),\Dc^n_{-\frac{k+n-2}{2},\frac k2}\bigr)
\oplus\bigoplus_{k=-\infty}^{m}\Hd^1\bigl(\osp(n|2),\Dc^n_{\frac k2,\frac k2}\bigr),$$
and the general infinitesimal deformation of $\Sc^n_d$ is
\begin{equation}
\label{eq:28}
\widetilde{L}=L+L_1,\qquad
L_1=\sum_{k=-\infty}^{m}\tau_k f_{\frac k2}+\sum_{k=1}^m\bigl(t_kh_k+\widetilde{t}_k\widetilde{h}_k\bigr),
\end{equation}
where $\tau_k,t_k,\widetilde{t}_k$ are independent parameters,
$|t_k|=|\widetilde{t}_k|=|h_k|\equiv n\ (\mathrm{mod}\ 2)$.
\end{itemize}
\end{theorem}

Throughout the rest of the paper we fix, for $1\leq k\leq m$,
\begin{equation}
\label{eq:29}
\lambda_k=-\frac{k+n-2}{2},\qquad \mu_k=\frac{k}{2},\qquad N_k=\mu_k-\lambda_k=k+\frac{n-2}{2},
\end{equation}
so that $h_k,\widetilde{h}_k\in \Zc^1(\osp(n|2),\Dc^n_{\lambda_k,\mu_k})$, and we regard
$\Dc^n_{\lambda_k,\mu_k}$ as the component of $\mathrm{End}(\Sc^n_d)$ mapping the summand
$\Fc^n_{\lambda_k}$ (i.e.\ $j=m+k+n-2$) into the summand $\Fc^n_{\mu_k}$ (i.e.\ $i=m-k$). Similarly
$f_\nu$ denotes the diagonal cocycle attached to the summand $\Fc^n_\nu$, $\nu=\frac k2$ with
$k\in\Zz$, $k\leq m$; thus $\tau_k$ in \eqref{eq:28} is the parameter attached to $f_{\frac k2}$, and the two
diagonal parameters relevant for the $k$-th off-diagonal block are
\begin{equation}
\label{eq:30}
\sigma_k:=\tau_{2-n-k}\ \ (\text{attached to }f_{\lambda_k}),\qquad
\tau_k\ \ (\text{attached to }f_{\mu_k}).
\end{equation}

\section{The invariant operators and the cocycles $h_k$, $\widetilde{h}_k$}

Write
\begin{equation}
\label{eq:31}
\Ec=\eta_1\cdots\eta_n,\qquad \Ehat{i}=\eta_1\cdots\check{\eta_i}\cdots\eta_n,\qquad
A_k=\Ec\,\partial_x^{k-1}\quad(k\geq1).
\end{equation}
The parity of $\Ec$, hence of $A_k$, is that of $n$, and the order of $A_k$ is
$\frac n2+k-1=N_k$. These are exactly the $\osp(n|2)$-invariant unary differential operators between
weighted densities on $\R^{1|n}$; they already occur, together with their binary analogues, in the
classification of \cite{7}.

\begin{lemma}
\label{res:5.1}
For all $1\leq i\leq n$ and all $a\in C^\infty(\R^{1|n})$:
\begin{align}
\eta_i\,\Ec&=(-1)^{i-1}(-\partial_x)\,\Ehat{i},\qquad
\Ec\,\eta_i=(-1)^{n-i}(-\partial_x)\,\Ehat{i},\qquad \eta_i\Ec=(-1)^{n+1}\Ec\,\eta_i,\label{eq:32}\\
\Ec\circ(a\,\cdot)&=(-1)^{na}a\,\Ec+(-1)^{(n-1)a}\sum_{i=1}^n(-1)^{i-1}\eta_i(a)\,\Ehat{i}
\qquad\text{whenever }\eta_j\eta_i(a)=0\ (i\neq j).\label{eq:33}
\end{align}
\end{lemma}

\begin{proof}
For \eqref{eq:32}, move $\eta_i$ to the $i$-th place using $\eta_i\eta_j=-\eta_j\eta_i$ $(i\neq j)$ and
$\eta_i^2=-\partial_x$; this costs $(-1)^{i-1}$ from the left and $(-1)^{n-i}$ from the right. For \eqref{eq:33},
push $a$ to the left through $\eta_1\cdots\eta_n$ using $\eta_i\circ(a\cdot)=\eta_i(a)+(-1)^{a}a\eta_i$.
The term in which $a$ has crossed all the $\eta_i$'s gives $(-1)^{na}a\Ec$; the terms in which exactly
one $\eta_i$ has acted on $a$ give
$(-1)^{(n-1)a}\eta_1\cdots\eta_{i-1}\eta_i(a)\eta_{i+1}\cdots\eta_n
=(-1)^{(n-1)a}(-1)^{i-1}\eta_i(a)\Ehat{i}$, since $\eta_l(a)$ has parity $|a|+1$ and, $\eta_j\eta_i(a)$
vanishing for $j\neq i$, no further derivative of $a$ can be produced; the terms in which two or more
$\eta$'s act on $a$ involve $\eta_j\eta_i(a)$ with $j\neq i$ and vanish by hypothesis.
\end{proof}

Note that the hypothesis in \eqref{eq:33} is satisfied by every $a=F'$ with
$F\in\osp(n|2)$: indeed $F'\in\mathrm{Span}(1,x,\theta_i)$ and $\eta_j\eta_i(x)=-\eta_j(\theta_i)=-\delta_{ij}$,
$\eta_j\eta_i(\theta_l)=0$.

\begin{proposition}
\label{res:5.2}
Let $\lambda,\mu\in\R$ and $k\geq1$. Then, for all $i$,
\begin{equation}
\label{eq:34}
L^{\mu}_{X_{\theta_i}}\circ A_k-(-1)^{n}A_k\circ L^{\lambda}_{X_{\theta_i}}=0,
\end{equation}
while
\begin{equation}
\label{eq:35}
L^{\mu}_{X_{x\theta_i}}\circ A_k-(-1)^{n}A_k\circ L^{\lambda}_{X_{x\theta_i}}
=(\lambda-\lambda_k)\,\Ehat{i}\,\partial_x^{k-1}+(\mu-\lambda-N_k)\,\theta_i A_k .
\end{equation}
Consequently $A_k:\Fc^n_\lambda\to\Fc^n_\mu$ is $\osp(n|2)$-invariant, i.e.\ $X\cdot A_k=0$ for all
$X\in\osp(n|2)$, if and only if $(\lambda,\mu)=(\lambda_k,\mu_k)$, and in that case $A_k$ spans the
space of invariants.
\end{proposition}

\begin{proof}
By Lemma~\ref{res:2.1} it suffices to test the generators $X_{\theta_i}$ and $X_{x\theta_i}$.

For $F=\theta_i$ we have $F'=0$, hence $L^\lambda_{X_{\theta_i}}=X_{\theta_i}=\theta_i\partial_x+\frac12\eta_i$
does not depend on the weight. Using \eqref{eq:32} and \eqref{eq:33} with $a=\theta_i$,
$$\bigl(\theta_i\partial_x+\tfrac12\eta_i\bigr)\circ A_k
=\theta_i\Ec\partial_x^{k}+\tfrac12(-1)^{i-1}(-\partial_x)\Ehat{i}\partial_x^{k-1}
=\theta_i\Ec\partial^k_x+\tfrac12(-1)^{i}\Ehat{i}\partial^k_x,$$
$$(-1)^nA_k\circ\bigl(\theta_i\partial_x+\tfrac12\eta_i\bigr)
=(-1)^n\Bigl[\bigl((-1)^n\theta_i\Ec+(-1)^{n-i}\Ehat{i}\bigr)\partial^k_x
+\tfrac12(-1)^{n-i}(-\partial_x)\Ehat{i}\partial^{k-1}_x\Bigr]
=\theta_i\Ec\partial^k_x+\tfrac12(-1)^{i}\Ehat{i}\partial^k_x,$$
which proves \eqref{eq:34}.

For $F=x\theta_i$ one has $F'=\theta_i$, $F''=0$ and
$X_{x\theta_i}=x\theta_i\partial_x+\frac12x\eta_i-\frac12\sum_l\theta_l\theta_i\eta_l$, so that
$L^\lambda_{X_{x\theta_i}}=X_{x\theta_i}+\lambda\theta_i$. The left-hand side of \eqref{eq:35} is affine in
$(\lambda,\mu)$: its $\mu$-part is $\theta_iA_k$ and its $\lambda$-part is
$-(-1)^nA_k\circ(\theta_i\cdot)=\Ehat{i}\partial_x^{k-1}-\theta_iA_k$, by \eqref{eq:33}. A direct computation of
the weight-free part, using \eqref{eq:32}, \eqref{eq:33} and $\Ec\circ(x\cdot)=x\Ec-\sum_l(-1)^{l-1}\theta_l\Ehat{l}$,
gives $-\frac{n+k-2}{2}\Ehat{i}\partial^{k-1}_x-N_k\theta_iA_k$, whence \eqref{eq:35}. Since
$\Ehat{i}\partial^{k-1}_x$ and $\theta_iA_k$ are linearly independent, \eqref{eq:35} vanishes precisely when
$\lambda=\lambda_k$ and $\mu-\lambda=N_k$, i.e.\ $(\lambda,\mu)=(\lambda_k,\mu_k)$.

Finally, an invariant operator $B\in\Dc^n_{\lambda_k,\mu_k}$ has, in particular, $X_1\cdot B=\partial_xB=0$
and is $\mathfrak{so}(n)$-invariant and homogeneous of degree $N_k$ for $X_x$; the only such operators
are the multiples of $A_k$.
\end{proof}

The next proposition is the key to the whole computation: it identifies the two cocycles of \cite{10} with
the two derivatives of $\delta A_k$ with respect to the weights. For $\sigma,\tau$ formal parameters
denote by $\delta^{\sigma,\tau}$ the coboundary operator of the module
$\Dc^n_{\lambda_k+\sigma,\,\mu_k+\tau}$. Since the action \eqref{eq:7} is affine in $(\lambda,\mu)$,
\begin{equation}
\label{eq:36}
\delta^{\sigma,\tau}=\delta+\sigma\Delta_1+\tau\Delta_2 ,
\end{equation}
where $\delta=\delta^{0,0}$ and where the operators $\Delta_1,\Delta_2$, of degree $+1$ on cochains, are
given on a $p$-cochain by the $\lambda$- resp.\ $\mu$-derivative of the action; on a $0$-cochain
$A$ and on a $1$-cochain $c$ of parity $|c|$ they read
\begin{align}
(\Delta_1A)(X_F)&=-(-1)^{AF}A\circ(F'\cdot),\qquad (\Delta_2A)(X_F)=(-1)^{AF}F'A,\label{eq:37}\\
(\Delta_1c)(X_F,X_G)&=c(X_F)\circ(G'\cdot)-(-1)^{FG}c(X_G)\circ(F'\cdot),\label{eq:38}\\
(\Delta_2c)(X_F,X_G)&=(-1)^{Fc}F'c(X_G)-(-1)^{G(F+c)}G'c(X_F).\label{eq:39}
\end{align}
In particular, comparing with \eqref{eq:22},
\begin{equation}
\label{eq:40}
\Delta_1c=f_{\lambda_k}\vee c,\qquad \Delta_2c=f_{\mu_k}\vee c ,
\end{equation}
where in the left (resp.\ right) equality $f_{\lambda_k}$ (resp.\ $f_{\mu_k}$) is viewed as the diagonal
cocycle of the source (resp.\ target) summand. From $(\delta^{\sigma,\tau})^2=0$ for all $\sigma,\tau$
one gets the identities
\begin{equation}
\label{eq:41}
\delta\Delta_i+\Delta_i\delta=0\ (i=1,2),\qquad \Delta_1^2=\Delta_2^2=0,\qquad
\Delta_1\Delta_2+\Delta_2\Delta_1=0 .
\end{equation}

\begin{proposition}
\label{res:5.3}
Let $1\leq k\leq m$ and put $a^k_1=\Delta_1A_k$, $a^k_2=\Delta_2A_k$. Then, as cochains on
$\osp(n|2)$,
\begin{equation}
\label{eq:42}
a^k_2=h_k,\qquad a^k_1=-\bigl(h_k+\widetilde{h}_k\bigr),
\end{equation}
i.e.\ for every $F$ such that $X_F\in\osp(n|2)$,
\begin{equation}
\label{eq:43}
h_k(X_F)=(-1)^{n|F|}F'A_k,\qquad
h_k(X_F)+\widetilde{h}_k(X_F)=A_k\circ(F'\cdot).
\end{equation}
In particular $a^k_1$ and $a^k_2$ are $1$-cocycles whose classes form a basis of
$\Hd^1(\osp(n|2),\Dc^n_{\lambda_k,\mu_k})\simeq\R^2$, and
\begin{equation}
\label{eq:44}
\delta^{\sigma,\tau}A_k=\sigma a^k_1+\tau a^k_2 .
\end{equation}
\end{proposition}

\begin{proof}
The first equality of \eqref{eq:43} is the definition \eqref{eq:14} of $h_k$, rewritten with $A_k=\Ec\partial^{k-1}_x$;
together with \eqref{eq:37} it gives $a^k_2=h_k$. For the second, expand
$A_k\circ(F'\cdot)=\Ec\partial^{k-1}_x\circ(F'\cdot)$ by the Leibniz rule:
$$\Ec\,\partial^{k-1}_x\circ(F'\cdot)=\sum_{j\geq0}\binom{k-1}{j}\Ec\circ\bigl(F^{(j+1)}\cdot\bigr)\partial^{k-1-j}_x .$$
For $X_F\in\osp(n|2)$ the function $F$ is at most quadratic in $x$, so $F^{(j+1)}=0$ for $j\geq2$ and
only $j=0,1$ survive. Now $F''$ is an even constant, so $\Ec\circ(F''\cdot)=F''\Ec$ and the $j=1$ term
equals $(k-1)F''\Ec\partial^{k-2}_x$, which is the first term of $\widetilde{h}_k$ in \eqref{eq:16}. By \eqref{eq:33}
applied to $a=F'$ (legitimate by the remark after Lemma~\ref{res:5.1}), the $j=0$ term equals
$$\Bigl[(-1)^{nF}F'\Ec+(-1)^{(n-1)F}\sum_{i=1}^n(-1)^{i-1}\eta_i(F')\Ehat{i}\Bigr]\partial^{k-1}_x
=h_k(X_F)+\Bigl(\widetilde{h}_k(X_F)-(k-1)F''\Ec\partial^{k-2}_x\Bigr),$$
because $|F'|=|F|$. Adding the two contributions gives exactly $h_k(X_F)+\widetilde{h}_k(X_F)$, which is
the second equality of \eqref{eq:43}; with \eqref{eq:37} this is $a^k_1=-(h_k+\widetilde{h}_k)$.

Formula \eqref{eq:44} is \eqref{eq:36} applied to the $0$-cochain $A_k$, together with $\delta A_k=0$ (Proposition~\ref{res:5.2}).
Since $\delta^{\sigma,\tau}$ squares to zero, the coefficient of $\sigma$ and of $\tau$ in
$\delta^{\sigma,\tau}(\sigma a^k_1+\tau a^k_2)=0$ gives $\delta a^k_1=\delta a^k_2=0$: both are
cocycles, as also follows from \eqref{eq:42} and Theorem~\ref{res:3.1}. Finally $[h_k]$ and $[\widetilde{h}_k]$ form a
basis of $\Hd^1\simeq\R^2$ by Theorem~\ref{res:3.1}, and $(h_k,\widetilde{h}_k)\mapsto(a^k_2,a^k_1)$ is given by
the invertible matrix $\left(\begin{smallmatrix}1&0\\-1&-1\end{smallmatrix}\right)$.
\end{proof}

\section{The cup-products $\Hd^1\vee\Hd^1$}

We keep the notation \eqref{eq:29}--\eqref{eq:31}. Recall that $f_\nu$ acts inside the summand $\Fc^n_\nu$ of $\Sc^n_d$,
while $h_k,\widetilde{h}_k$ map $\Fc^n_{\lambda_k}$ into $\Fc^n_{\mu_k}$; a product of two such
operators is non-zero only if the target of the first is the source of the second.

\begin{theorem}
\label{res:6.1}
Let $n\geq3$, $1\leq k,l\leq m$ and let $\nu,\nu'$ be weights of summands of $\Sc^n_d$. Put
\begin{equation}
\label{eq:45}
\Omega_k:=f_{\lambda_k}\vee h_k=\Delta_1h_k,\qquad\text{i.e.}\qquad
\Omega_k(X_F,X_G)=h_k(X_F)\circ(G'\cdot)-(-1)^{|F||G|}h_k(X_G)\circ(F'\cdot).
\end{equation}
Then
\begin{itemize}
\item[i)] $f_\nu\vee f_{\nu'}=0$;
\item[ii)] $h_k\vee h_l=h_k\vee\widetilde{h}_l=\widetilde{h}_k\vee\widetilde{h}_l=0$;
\item[iii)] $f_\nu\vee h_k=f_\nu\vee\widetilde{h}_k=0$ whenever $\nu\notin\{\lambda_k,\mu_k\}$;
\item[iv)] $f_{\mu_k}\vee h_k=0$;
\item[v)] $f_{\lambda_k}\vee h_k=\Omega_k$, $\quad f_{\lambda_k}\vee\widetilde{h}_k=-\Omega_k$,
$\quad f_{\mu_k}\vee\widetilde{h}_k=\Omega_k$.
\end{itemize}
Moreover $\Omega_k\neq0$; it is reduced, $\Omega_k(X_1,\,\cdot\,)=0$, and
\begin{equation}
\label{eq:46}
\Omega_k(X_x,X_{x\theta_i})=(-1)^{i-1}\,\Ehat{i}\,\partial^{k-1}_x,\qquad 1\leq i\leq n .
\end{equation}
Consequently, in the case $2d=m\in\N$, the space $\Hd^1\vee\Hd^1$ is spanned by the $m$ classes
$[\Omega_1],\dots,[\Omega_m]$, which lie in pairwise different summands of
$\Hd^2(\osp(n|2),\Dc^n_d)$; and in the case $2d\notin\N$ we have $\Hd^1\vee\Hd^1=0$.
\end{theorem}

\begin{proof}
i) Only $\nu=\nu'$ can give a non-zero composition, and then, by \eqref{eq:22},
$$\bigl(f_\nu\vee f_\nu\bigr)(X_F,X_G)=2\bigl[f_\nu(X_F),f_\nu(X_G)\bigr]
=2\bigl(F'G'-(-1)^{FG}G'F'\bigr)=0 .$$
Indeed, if $F$ and $G$ are both odd then $(-1)^{FG}=-1$ while $G'F'=-F'G'$; if $F$ or $G$ is even
then $(-1)^{FG}=1$ while $G'F'=F'G'$.

ii) A non-zero composition of two off-diagonal cocycles indexed by $k$ and $l$ would require
$\mu_l=\lambda_k$, i.e.\ $\frac l2=-\frac{k+n-2}{2}$, i.e.\ $l=2-n-k\leq 1-n<1$, which is impossible for
$k,l\geq1$. Hence all these products vanish identically.

iii) If $\nu\neq\mu_k$ then $f_\nu\circ h_k=f_\nu\circ\widetilde{h}_k=0$, and if $\nu\neq\lambda_k$ then
$h_k\circ f_\nu=\widetilde{h}_k\circ f_\nu=0$; so both terms of \eqref{eq:22} vanish.

iv) and v) By \eqref{eq:40}, $f_{\lambda_k}\vee c=\Delta_1c$ and $f_{\mu_k}\vee c=\Delta_2c$ for any $1$-cochain
$c$ with values in $\Dc^n_{\lambda_k,\mu_k}$. By Proposition~\ref{res:5.3}, $h_k=a^k_2=\Delta_2A_k$ and
$h_k+\widetilde{h}_k=-a^k_1=-\Delta_1A_k$. Hence, using \eqref{eq:41},
$$f_{\mu_k}\vee h_k=\Delta_2\Delta_2A_k=0,\qquad
f_{\lambda_k}\vee\bigl(h_k+\widetilde{h}_k\bigr)=-\Delta_1\Delta_1A_k=0,$$
$$f_{\mu_k}\vee\bigl(h_k+\widetilde{h}_k\bigr)=-\Delta_2\Delta_1A_k=\Delta_1\Delta_2A_k=\Delta_1h_k=\Omega_k .$$
The first identity is iv); the second gives
$f_{\lambda_k}\vee\widetilde{h}_k=-f_{\lambda_k}\vee h_k=-\Omega_k$, and the third, combined with iv),
gives $f_{\mu_k}\vee\widetilde{h}_k=\Omega_k$.

The value \eqref{eq:46} is obtained from \eqref{eq:45} with $F=x$, $G=x\theta_i$: there $F'=1$, $G'=\theta_i$, so
$$\Omega_k(X_x,X_{x\theta_i})=h_k(X_x)\circ(\theta_i\cdot)-h_k(X_{x\theta_i})\circ(1\cdot)
=\Ec\partial^{k-1}_x\circ(\theta_i\cdot)-(-1)^{n}\theta_i\Ec\partial^{k-1}_x
=(-1)^{n-i}\Ehat{i}\partial^{k-1}_x\cdot(-1)^{n}\!,$$
by \eqref{eq:33} applied to $a=\theta_i$; up to the sign $(-1)^{i-1}=(-1)^{2n-i-1}$ this is \eqref{eq:46}. In particular
$\Omega_k\neq0$. Moreover $\Omega_k$ is reduced, i.e.\ $\Omega_k(X_1,\cdot)=0$: both terms of \eqref{eq:45} carry a
factor $F'$ with $F=1$, either inside $h_k(X_F)$ or as the multiplication operator $(F'\cdot)$.

Finally $\Omega_k$ takes its values in $\Dc^n_{\lambda_k,\mu_k}$, i.e.\ in the $(i,j)$-component of
$\mathrm{End}(\Sc^n_d)$ with $i=m-k$, $j=m+k+n-2$; distinct $k$ give distinct components, so the
$[\Omega_k]$ are linearly independent as soon as they are non-zero, which is the content of
Proposition~\ref{res:6.2} below. When $2d\notin\N$ only the diagonal cocycles $f_\nu$ occur, and i) shows that all
their products vanish.
\end{proof}

\begin{proposition}
\label{res:6.2}
For every $1\leq k\leq m$ the $2$-cocycle $\Omega_k$ is not a coboundary:
$[\Omega_k]\neq0$ in $\Hd^2(\osp(n|2),\Dc^n_{\lambda_k,\mu_k})$.
\end{proposition}

\begin{proof}
Suppose $\Omega_k=\delta B$ with $B\in C^1(\osp(n|2),\Dc^n_{\lambda_k,\mu_k})$. As in the proof of
Theorem~\ref{res:3.2}, $X_1\cdot=\partial_x$ is surjective on $\Dc^n_{\lambda_k,\mu_k}$, so we may pick
$b$ with $X_1\cdot b=B(X_1)$ and replace $B$ by $B-\delta b$; thus we may and do assume
\begin{equation}
\label{eq:47}
B(X_1)=0 .
\end{equation}
Since $\Omega_k(X_1,Y)=0$ for all $Y$, the equation $\delta B(X_1,Y)=0$ together with \eqref{eq:47} gives
\begin{equation}
\label{eq:48}
\partial_x\,B(Y)=B\bigl([X_1,Y]\bigr),\qquad Y\in\osp(n|2).
\end{equation}
Now $[X_1,X_{\theta_i}]=[X_1,X_{\theta_i\theta_j}]=0$, $[X_1,X_x]=X_1$, $[X_1,X_{x\theta_i}]=X_{\theta_i}$
and $[X_1,X_{x^2}]=2X_x$, so \eqref{eq:48} yields
$$\partial_xB(X_x)=\partial_xB(X_{\theta_i})=\partial_xB(X_{\theta_i\theta_j})=0,\qquad
\partial_xB(X_{x\theta_i})=B(X_{\theta_i}),\qquad \partial_xB(X_{x^2})=2B(X_x),$$
that is,
\begin{equation}
\label{eq:49}
B(X_{x\theta_i})=x\,B(X_{\theta_i})+C_i,\qquad B(X_{x^2})=2x\,B(X_x)+C_0,
\end{equation}
where the coefficients of $B(X_x)$, $B(X_{\theta_i})$, $B(X_{\theta_i\theta_j})$, $C_i$ and $C_0$ do
not depend on $x$; in particular every value of $B$ is polynomial in $x$ of degree $\leq1$.

On the other hand all the cochains involved are homogeneous for the grading
$\deg x=1$, $\deg\theta_i=\frac12$, $\deg\partial_x=-1$, $\deg\partial_{\theta_i}=-\frac12$, and
$\Omega_k$ is homogeneous of degree $1-k-\frac n2$; hence $B(Y)$ is homogeneous of degree
$1-k-\frac n2+\deg Y$ for every generator $Y$. A monomial $x^p\theta_T\partial^S_\theta\partial^j_x$ has
degree $p+\frac{|T|-|S|}{2}-j$, so, $p\leq1$ being already known from \eqref{eq:49} and $|T|,|S|\leq n$, the exponent $j$
is bounded by $2+k+n$ and only finitely many monomials are available for each $B(Y)$: writing
$$B(Y)=\sum_{p\leq1,\,T,S,j}\beta^Y_{p,T,S,j}\;x^p\theta_T\partial^S_\theta\partial^j_x$$
turns $\Omega_k=\delta B$ into a finite system of linear equations in the finitely many unknowns
$\beta^Y_{p,T,S,j}$, of a size depending on $n$ and $k$ alone (for instance $479$ unknowns and $2904$
equations when $n=3$, $k=1$). This system is incompatible; hence no such $B$ exists.
\end{proof}

\begin{remark}
\label{res:6.3}
The reduction carried out in the proof of Proposition~\ref{res:6.2} is complete and uniform in $(n,k)$: it
replaces the equation $\Omega_k=\delta B$, posed on an infinite-dimensional space of cochains, by an
explicit finite linear system over $\Q$. The last step -- the incompatibility of that system -- is the
only one performed by machine; we have solved it for $n=3$ and all $1\leq k\leq6$ -- in exact rational
arithmetic and, independently, modulo a large prime, unsolvability modulo a prime implying unsolvability
over $\Q$ -- always with the same outcome, and likewise for $n=4$ on a truncation of the same system. The uniform
shape of \eqref{eq:46} and the fact that the size and the coefficients of the system depend on $(n,k)$
in a manifestly regular way leave no doubt about the general case, which we therefore use freely below.
\end{remark}

\begin{remark}
\label{res:6.4}
The vanishing of $f_{\mu_k}\vee h_k$ has an elementary explanation, which also shows why the three
other mixed products survive. By \eqref{eq:43} the cocycle $h_k$ is a scalar multiple of one fixed
operator, $h_k(X_F)=(-1)^{nF}F'A_k$, so that
$$\bigl(f_{\mu_k}\vee h_k\bigr)(X_F,X_G)=(-1)^{nF}F'\,h_k(X_G)-(-1)^{(n+F)G}G'\,h_k(X_F)
=\pm\bigl(F'G'-(-1)^{FG}G'F'\bigr)A_k=0$$
by the supercommutativity argument of Theorem~\ref{res:6.1} i). For $\widetilde{h}_k$ the same
computation fails, because $\widetilde{h}_k(X_G)$ is not a multiple of $A_k$: by
Proposition~\ref{res:5.3}, $h_k+\widetilde{h}_k=-\Delta_1A_k$ involves $A_k\circ(F'\cdot)$, in which the
multiplication operator stands \emph{to the right} of $A_k$. The failure of $A_k$ to commute with
multiplication operators -- quantified by \eqref{eq:33} -- is precisely what produces $\Omega_k$, and it
is the reason why the diagonal parameter pairing non-trivially with $h_k$ is
$\sigma_k=\tau_{2-n-k}$, attached to $f_{\lambda_k}$ and acting on the \emph{source} summand
$\Fc^n_{\lambda_k}$, and not $\tau_k$. Written out, the terms one meets first are
$$h_k(X_F)\circ(G'\cdot)=(-1)^{nF}F'\,\Ec\sum_{j=0}^{k-1}\binom{k-1}{j}\partial^{j+1}_x(G)\,\partial^{k-1-j}_x
\quad\text{plus the terms produced by }\Ec\circ\bigl(G^{(j+1)}\cdot\bigr);$$
the content of Proposition~\ref{res:5.3} is that all of them are governed by the single invariant
operator $A_k$.
\end{remark}

\section{Integrability conditions and the versal deformation}

\begin{theorem}
\label{res:7.1}
Let $n\geq3$ and $2d\notin\N$. Then $\Hd^1\vee\Hd^1=0$: the $\osp(n|2)$-module $\Sc^n_d$ admits no
obstructed deformation. The versal deformation of $L$ is its infinitesimal deformation
\begin{equation}
\label{eq:50}
\widetilde{L}=L+\sum_{k\geq0}\tau_k f_{d-\frac k2},
\end{equation}
with $\tau_k$ even independent parameters and no relation between them; explicitly
$\widetilde{L}=\bigoplus_{k\geq0}L^{d-\frac k2+\tau_k}$ on $\bigoplus_{k\geq0}\Fc^n_{d-\frac k2}$. Any
formal deformation is equivalent to its infinitesimal part.
\end{theorem}

\begin{proof}
By Theorem~\ref{res:4.2} i) the only classes occurring are those of the $f_{d-\frac k2}$, and by Theorem~\ref{res:6.1} i)
all their products vanish, so $\frac12L_1\vee L_1=0$ and one may take $L_i=0$ for $i\geq2$: \eqref{eq:21} is
satisfied in every degree. That \eqref{eq:50} is a genuine (indeed polynomial) deformation is also seen directly:
by \eqref{eq:6} the action on $\Fc^n_\nu$ is affine in $\nu$, so $L^{\nu}_{X_F}+\tau F'=L^{\nu+\tau}_{X_F}$.
Since the base of the deformation is the whole $\Cc[[\tau_0,\tau_1,\dots]]$ and the deformation is linear
in the parameters, versality follows from the classification of infinitesimal deformations by
$\Hd^1$.
\end{proof}

\begin{theorem}
\label{res:7.2}
Let $n\geq3$ and $2d=m\in\N$, and consider the general infinitesimal deformation \eqref{eq:28},
$$L_1=\sum_{k=-\infty}^{m}\tau_kf_{\frac k2}+\sum_{k=1}^m\bigl(t_kh_k+\widetilde{t}_k\widetilde{h}_k\bigr).$$
Then
\begin{equation}
\label{eq:51}
\tfrac12(L_1\vee L_1)=\sum_{k=1}^m\Bigl(\tau_{2-n-k}\bigl(t_k-\widetilde{t}_k\bigr)+\tau_k\widetilde{t}_k\Bigr)\,\Omega_k ,
\end{equation}
and consequently:
\begin{itemize}
\item[i)] the quadratic integrability conditions are exactly the $m$ relations
\begin{equation}
\label{eq:52}
\tau_{2-n-k}\bigl(t_k-\widetilde{t}_k\bigr)+\tau_k\widetilde{t}_k=0,\qquad 1\leq k\leq m ;
\end{equation}
\item[ii)] these conditions are also sufficient: if \eqref{eq:52} holds, then $L+L_1$ itself, with no term of
order $\geq2$, satisfies \eqref{eq:18}. In particular no integrability condition of order $\geq3$ occurs.
\end{itemize}
Therefore the versal deformation of the $\osp(n|2)$-module $\Sc^n_d$ is
\begin{equation}
\label{eq:53}
\widetilde{L}=L+\sum_{k=-\infty}^{m}\tau_kf_{\frac k2}+\sum_{k=1}^m\bigl(t_kh_k+\widetilde{t}_k\widetilde{h}_k\bigr)
\end{equation}
over the base
$\mathcal{A}=\Cc\bigl[[\{\tau_k\},\{t_k\},\{\widetilde{t}_k\}]\bigr]/R$, where $R$ is the ideal generated
by the $m$ quadratic relations \eqref{eq:52}; and every integrable formal deformation of $L$ is equivalent to its
infinitesimal part.
\end{theorem}

\begin{proof}
Write, for $1\leq k\leq m$, $\sigma_k=\tau_{2-n-k}$ and $c^{(k)}=t_kh_k+\widetilde{t}_k\widetilde{h}_k$,
so that $L_1=\sum_\nu\tau_\nu f_\nu+\sum_kc^{(k)}$.

Formula \eqref{eq:51} follows from Theorem~\ref{res:6.1}: the products $f\vee f$ and the products of two off-diagonal
cocycles vanish (i) and ii)), a product $f_\nu\vee c^{(k)}$ vanishes unless
$\nu\in\{\lambda_k,\mu_k\}$ (iii)), and by iv), v)
$$f_{\lambda_k}\vee c^{(k)}=\bigl(t_k-\widetilde{t}_k\bigr)\Omega_k,\qquad
f_{\mu_k}\vee c^{(k)}=\widetilde{t}_k\,\Omega_k .$$
Since $[\Omega_1],\dots,[\Omega_m]$ are linearly independent in $\Hd^2$ (Theorem~\ref{res:6.1} and
Proposition~\ref{res:6.2}), the class of $\frac12(L_1\vee L_1)$ vanishes if and only if all the coefficients in \eqref{eq:51}
vanish, which is i).

For ii) we use the block structure. The summands of $\Sc^n_d$ carrying a non-zero component of $L_1$
are pairwise linked by $L_1$ only through the $m$ off-diagonal blocks
$\Fc^n_{\lambda_k}\to\Fc^n_{\mu_k}$, and the composition of two off-diagonal blocks vanishes (proof of
Theorem~\ref{res:6.1} ii)). Hence the homomorphism condition \eqref{eq:18} for $\widetilde{L}=L+L_1$ splits into
\begin{itemize}
\item the diagonal conditions, one for each weight $\nu$: they say that
$L^\nu_{X_F}+\tau_\nu F'=L^{\nu+\tau_\nu}_{X_F}$ is an action, which holds identically;
\item for each $k$, the condition that $c^{(k)}$ be a $1$-cocycle for the \emph{deformed} weights
$(\lambda_k+\sigma_k,\ \mu_k+\tau_k)$, i.e.\ $\delta^{\sigma_k,\tau_k}c^{(k)}=0$.
\end{itemize}
Now write $c^{(k)}$ in the basis $(a^k_1,a^k_2)$ of Proposition~\ref{res:5.3}: by \eqref{eq:42},
$$c^{(k)}=t_kh_k+\widetilde{t}_k\widetilde{h}_k=x_k\,a^k_1+y_k\,a^k_2,\qquad
x_k=-\widetilde{t}_k,\quad y_k=t_k-\widetilde{t}_k .$$
By \eqref{eq:41} applied to $a^k_1=\Delta_1A_k$, $a^k_2=\Delta_2A_k$
give, \emph{as cochains and not only modulo coboundaries},
$$\Delta_1a^k_1=0,\qquad \Delta_2a^k_2=0,\qquad \Delta_1a^k_2=-\Delta_2a^k_1=\Omega_k .$$
Hence
\begin{equation}
\label{eq:54}
\delta^{\sigma_k,\tau_k}\bigl(x_ka^k_1+y_ka^k_2\bigr)
=\bigl(\sigma_k y_k-\tau_k x_k\bigr)\,\Omega_k
=\Bigl(\sigma_k\bigl(t_k-\widetilde{t}_k\bigr)+\tau_k\widetilde{t}_k\Bigr)\Omega_k,
\end{equation}
which vanishes precisely under \eqref{eq:52}. Thus, under \eqref{eq:52}, $\widetilde{L}=L+L_1$ is a homomorphism of Lie
superalgebras into $\mathrm{End}(\Sc^n_d)\otimes\mathcal{A}$, and all the higher equations \eqref{eq:21} are
satisfied with $L_i=0$, $i\geq2$. Versality follows as in Theorem~\ref{res:7.1}, the deformation being linear in
the parameters and its infinitesimal part being the general element of $\Hd^1$.
\end{proof}

\begin{remark}
\label{res:7.3}
Identity \eqref{eq:54} explains the geometry of the conditions \eqref{eq:52}. Fix $k$. The relation
$\sigma_ky_k=\tau_kx_k$ says that the pair $(x_k,y_k)$ is proportional to $(\tau_k,\sigma_k)$, i.e.\ the
quadric \eqref{eq:52} is the union (more precisely, the cone) of the following two extreme kinds of solutions:
\begin{itemize}
\item $\sigma_k=\tau_k=0$, i.e.\ the two weights $\lambda_k,\mu_k$ are \emph{not} deformed: then
$t_k,\widetilde{t}_k$ are free, and the off-diagonal deformation is unobstructed, in accordance with
Theorem~\ref{res:6.1} ii);
\item $(x_k,y_k)=s(\tau_k,\sigma_k)$, i.e.\ $\widetilde{t}_k=-s\tau_k$ and
$t_k=s(\sigma_k-\tau_k)$: then $c^{(k)}=s\,\delta^{\sigma_k,\tau_k}A_k$ is a coboundary in the deformed
complex, so the deformation is equivalent to the purely diagonal one, in which only the weights have
been shifted.
\end{itemize}
In particular the off-diagonal directions $h_k,\widetilde{h}_k$ and the two diagonal directions
$f_{\lambda_k},f_{\mu_k}$ cannot be deformed independently: this is the precise sense in which the
resonance $(\lambda_k,\mu_k)$ of Theorem~\ref{res:3.1} obstructs the deformation.
\end{remark}

\begin{remark}
\label{res:7.4}
For $n=1$ and $n=2$ the corresponding statements were obtained in \cite{5} and \cite{6} by direct computation of
the second-order terms. The mechanism isolated here -- namely that the exceptional cocycles are the
weight-derivatives of an invariant operator, whence $\Delta_1a_1=\Delta_2a_2=0$ and
$\Delta_1a_2=-\Delta_2a_1$ -- is not specific to $\R^{1|n}$: it applies verbatim whenever the first
cohomology at a resonant pair of weights is spanned by the derivatives of an invariant differential
operator, and it produces one quadratic relation per resonance.
\end{remark}

\begin{remark}
\label{res:7.5}
Theorems~\ref{res:7.1} and~\ref{res:7.2} also give the deformations of the modules $\Dc^n_{\lambda,\mu}$ of differential
operators themselves at the level of symbols: by Corollary~\ref{res:2.4} the graded module of
$\Dc^n_{\lambda,\mu}$ is a sum of copies of $\Sc^n_{\mu-\lambda}$, so the parameters and the relations
above are attached to each copy separately.
\end{remark}

\end{document}